\documentclass[11pt]{amsart}
\usepackage{latexsym,amsmath,amssymb,amsthm,mathrsfs,amscd,enumerate,latexsym,array}
\usepackage[all]{xy}

\theoremstyle{definition}
\newtheorem{Unity}{Unity}[section]
\newtheorem{dfn}[Unity]{Definition}
\newtheorem{rmk}[Unity]{Remark}

\newtheorem{const}[Unity]{Construction}
\newtheorem{ass}[Unity]{Assumption}

\newtheorem{case1}{Case}
\newtheorem{case12}{Case}[case1]
\newtheorem{case13}{Case}[case12]

\newtheorem{claim}{Claim}
\newtheorem{case2}{Case}[claim]
\newtheorem{case22}{Case}[case2]

\newtheorem*{ack}{Acknowledgements}

\theoremstyle{plain}
\newtheorem{thm}[Unity]{Theorem}
\newtheorem{prop}[Unity]{Proposition}
\newtheorem{conj}[Unity]{Conjecture}
\newtheorem{lem}[Unity]{Lemma}
\newtheorem{cor}[Unity]{Corollary}

\let\oldenumerate\enumerate
\renewcommand{\enumerate}{
\oldenumerate
\setlength{\itemsep}{3pt}
\setlength{\leftskip}{-13pt}
}

\let\olditemize\itemize
\renewcommand{\itemize}{
\olditemize
\setlength{\itemsep}{2pt}
\setlength{\leftskip}{-15pt}
}

\begin{document}

\title{On the Picard number of Fano 6-folds with a non-small contraction}
\author{Taku Suzuki}
\keywords{Mukai's Conjecture; Fano manifolds; Picard number; pseudo-index; rational curves.}
\subjclass[2010]{14J45, 14E30.}
\address{Department of Mathematics, School of Fundamental Science and Engineering, Waseda University, 3-4-1 Ohkubo, Shinjuku, Tokyo 169-8555, Japan}
\email{s.taku.1231@moegi.waseda.jp}

\maketitle

\begin{abstract}
A generalization of S.\ Mukai's conjecture says that if $X$ is a Fano $n$-fold with Picard number $\rho_X$ and pseudo-index $i_X$, then $\rho_X(i_X-1) \leq n$, with equality if and only if $X \cong (\mathbb{P}^{i_X-1})^{\rho_X}$. 
In this paper, we prove that this conjecture holds if $n=6$ and either $X$ admits a contraction of fiber type or $X$ admits no small contractions. 
\end{abstract}

\section{Introduction}
We consider Fano manifolds, i.e., smooth complex projective varieties with ample anti-canonical line bundle. 
In 1988, S.\ Mukai proposed the following conjecture: 

\begin{conj}[{\cite{Mu}}]\label{Mukai}
Let $X$ be a Fano $n$-fold with Picard number $\rho _X$ and index $r_X$ (see Definition \ref{index}). 
Then $\rho _X (r_X-1) \leq n$, with equality if and only if $X \cong (\mathbb{P}^{r_X-1})^{\rho_X}$. 
\end{conj}

Conjecture \ref{Mukai} for $r_X \geq n+1$ corresponds to a characterization of projective spaces by S.\ Kobayashi and T.\ Ochiai (\cite{KO}). 

L.\ Bonavero, C.\ Casagrande, O.\ Debarre, and S.\ Druel proposed the following more general conjecture: 

\begin{conj}[{\cite{BCDD}}]\label{GMukai}
Let $X$ be a Fano $n$-fold with Picard number $\rho _X$ and pseudo-index $i_X$ (see Definition \ref{pindex}). 
Then $\rho _X (i_X-1) \leq n$, with equality if and only if $X \cong (\mathbb{P}^{i_X-1})^{\rho_X}$. 
\end{conj}

Conjecture \ref{GMukai} for $i_X \geq n+1$ corresponds to a characterization of projective spaces by K.\ Cho, Y.\ Miyaoka, and N.\ I.\  Shepherd-Barron (\cite{CMSB}). 
Conjecture \ref{GMukai} has been investigated by many authors, and it has been proved in the following cases: 
\begin{itemize}
\item In 2003, by L.\ Bonavero, C.\ Casagrande, O.\ Debarre, and S.\ Druel (\cite{BCDD}), 
\begin{itemize}
\item $n \leq 4$,
\item $X$ is homogeneous,
\item $X$ is toric and either $i_X \geq \frac{n+3}{3}$ or $n \leq 7$.
\end{itemize}
\item In 2004, by M.\ Andreatta, E.\ Chierici, and G.\ Occhetta (\cite{ACO}), 
\begin{itemize}
\item $n = 5$,
\item $i_X \geq \frac{n+3}{3}$ and $X$ admits an unsplit dominating family of rational curves (see Definition \ref{unsplit}).
\end{itemize}
\item In 2006, by C.\ Casagrande (\cite{Ca1}), 
\begin{itemize}
\item $X$ is toric.
\end{itemize}
\item In 2010, by C.\ Novelli and G.\ Occhetta (\cite{NO}), 
\begin{itemize}
\item $i_X \geq \frac{n+3}{3}$.
\end{itemize}
\item In 2012, by C.\ Novelli (\cite{No}), 
\begin{itemize}
\item $i_X = \frac{n+1}{3}, \frac{n+2}{3}$, 
\item $i_X = \frac{n}{3}$ and $X$ admits an unsplit dominating family of rational curves. 
\end{itemize}
\end{itemize}

By \cite{NO} and \cite{No}, we know that Conjecture \ref{GMukai} for $n=6$ is true if either $i_X \geq 3$ or $X$ admits an unsplit dominating family of rational curves. 
In this paper, we consider Fano $6$-folds which admit no unsplit dominating families of rational curves, and we prove the following theorem: 

\begin{thm}\label{main}
Let $X$ be a Fano $6$-fold with Picard number $\rho _X$ and pseudo-index $i_X \geq 2$. 
Assume that $X$ admits no unsplit dominating families of rational curves. 
\begin{enumerate}
\item If $X$ admits a contraction of fiber type, then $\rho_X \leq 2$ holds. 
\item If $X$ admits no small contractions, then $\rho_X \leq 4$ holds. 
\end{enumerate}
In particular, Conjecture \ref{GMukai} for $n=6$ is true if either $X$ admits a contraction of fiber type or $X$ admits no small contractions. 
\end{thm}

\begin{ack}
The author would like to express his gratitude to his supervisor Professor Hajime Kaji for valuable discussions. 
The author would also like to thank Professor Yasunari Nagai and Professor Kazunori Yasutake for beneficial comments. 
The author is supported by Grant-in-Aid for Research Activity Start-up from the Japan Society for the Promotion of Science. 
\end{ack}

\section{Families of rational curves}

Throughout this section, we consider a smooth complex projective manifold $X$.  

\begin{dfn}
We denote by ${\rm RatCurves}^n(X)$ the normalization of the scheme of rational curves on $X$ (see \cite[II.2]{Ko}). 
A \textit{family of rational curves} on $X$ is an irreducible component of ${\rm RatCurves}^n(X)$. 
Given a rational curve $C$ on $X$, we define a \textit{family of deformations} of $C$ to be a family of rational curves containing $C$. 
\end{dfn}

\begin{dfn}
Let $\mathscr{V}$ be a family of rational curves. 
Let $\mathscr{U}$ be the universal family of $\mathscr{V}$, and let $p:\mathscr{U} \rightarrow \mathscr{V}$ and $q:\mathscr{U} \rightarrow X$ be the associated morphisms. 
$q(\mathscr{U})$ is denoted by ${\rm Locus}(\mathscr{V})$. 
We say that $\mathscr{V}$ is \textit{dominating} if $\overline{{\rm Locus}(\mathscr{V})}=X$. 

For a subvariety $Y \subset X$, $p(q^{-1}(Y))$ (the subscheme of $\mathscr{V}$ which parametrizes curves intersecting $Y$) is denoted by $\mathscr{V}_Y$, and $q(p^{-1}(\mathscr{V}_Y))$ is denoted by ${\rm Locus}(\mathscr{V}; Y)$. 
In particular, when $Y$ is a point, $\mathscr{V}_{\{x\}}$ (resp. ${\rm Locus}(\mathscr{V}; \{x\})$) is also denoted by $\mathscr{V}_x$ (resp. ${\rm Locus}(\mathscr{V}; x)$). 

For families $\mathscr{V}_1,$ $\ldots ,$ $\mathscr{V}_m$ of rational curves, we inductively define 
$${\rm Locus}(\mathscr{V}_m, \ldots , \mathscr{V}_1; Y):={\rm Locus}(\mathscr{V}_m; {\rm Locus}(\mathscr{V}_{m-1}, \ldots , \mathscr{V}_1; Y)).$$
When $\mathscr{V}_1= \cdots = \mathscr{V}_m=\mathscr{V}$, we denote it by 
$${\rm Locus}_m(\mathscr{V}; Y):={\rm Locus}(\underbrace{\mathscr{V}, \ldots , \mathscr{V}}_{m} ; Y).$$
\end{dfn}

\begin{dfn}
For a family $\mathscr{V}$ of rational curves, we denote by $\overline{\mathscr{V}}$ the closure of $\mathscr{V}$ in ${\rm Chow}(X)$. 
Note that $\overline{\mathscr{V}}$ parametrizes connected 1-cycles whose every component is a rational curve. 
We similarly define ${\rm Locus}(\overline{\mathscr{V}})$, $\overline{\mathscr{V}}_Y$, ${\rm Locus}(\overline{\mathscr{V}}; Y)$, ${\rm Locus}(\overline{\mathscr{V}_m}, \ldots , \overline{\mathscr{V}_1}; Y)$, and ${\rm Locus}_m(\overline{\mathscr{V}}; Y)$. 
\end{dfn}

\begin{dfn}\label{unsplit}
Let $\mathscr{V}$ be a family of rational curves. 
\begin{enumerate}
\item $\mathscr{V}$ is said to be \textit{unsplit} if it is proper (i.e., $\overline{\mathscr{V}}=\mathscr{V}$ as sets). 
\item For a closed subvariety $Y \subset X$ intersecting ${\rm Locus}(\mathscr{V})$ (resp.\ a point $x \in {\rm Locus}(\mathscr{V})$), $\mathscr{V}$ is said to be \textit{unsplit on $Y$} (resp.\ \textit{unsplit at $x$}) if $\mathscr{V}_Y$ (resp.\ $\mathscr{V}_x$) is proper. 
\item $\mathscr{V}$ is said to be \textit{locally unsplit} if  $\mathscr{V}$ is unsplit at a general point of ${\rm Locus}(\mathscr{V})$. 
\item $\mathscr{V}$ is said to be \textit{quasi-unsplit} if every component of any cycle parametrized by $\overline{\mathscr{V}}$ is numerically proportional to $\mathscr{V}$. 
\end{enumerate}
\end{dfn}

\begin{dfn}\label{chain}
Let $\mathscr{V}_1,$ $\ldots ,$ $\mathscr{V}_m$ be families of rational curves. 
We define a \textit{connected chain parametrized by} $(\overline{\mathscr{V}_1}, \ldots , \overline{\mathscr{V}_m})$ as a chain $\{\Gamma_i \}_{i=1}^m$ such that each $\Gamma_i$ is a cycle parametrized by $\overline{\mathscr{V}_i}$ and $\Gamma_i \cap \Gamma_{i+1}\ne \emptyset$. 
We say that two subvarieties $Y,Z \subset X$ \textit{can be connected by} $(\overline{\mathscr{V}_1}, \ldots , \overline{\mathscr{V}_m})$ if there exists a connected chain $\{\Gamma_i \}_{i=1}^m$ parametrized by $(\overline{\mathscr{V}_1}, \ldots , \overline{\mathscr{V}_m})$ such that $Y\cap \Gamma_1\ne \emptyset$ and $\Gamma_m \cap Z\ne \emptyset$. 
\end{dfn}

\begin{dfn}
Let $\mathscr{V}_1,$ $\ldots ,$ $\mathscr{V}_k$ be families of rational curves. 
We say that two points $x, y \in X$ are in \textit{$m$-rc$(\mathscr{V}_1, \ldots , \mathscr{V}_k)$-relation} if they can be connected by $(\overline{\mathscr{V}_{i(1)}}, \ldots , \overline{{\mathscr{V}_{i(m)}}})$ for some $1 \leq i(j) \leq k$. 
We say that $x$ and $y$ are in \textit{rc$(\mathscr{V}_1, \ldots , \mathscr{V}_k)$-relation} if they are in $m$-rc$(\mathscr{V}_1, \ldots , \mathscr{V}_k)$-relation  for some $m$. 
It is known (\cite[IV Theorem 4.16]{Ko}) that there exists an open subvariety $U \subset X$, a proper surjective morphism $\pi : U \rightarrow Z$ with connected fibers, and an integer $m \le 2^{{\rm dim}\,X-{\rm dim}\,Z}-1$ such that
\begin{itemize}
\item[(i)] the rc$(\mathscr{V}_1, \ldots , \mathscr{V}_k)$-relation restricts to an equivalence relation on $U$,
\item[(ii)] for every $z \in Z$, $\pi ^{-1}(z)$ coincides with an equivalence class for the rc$(\mathscr{V}_1, \ldots , \mathscr{V}_k)$-relation,
\item[(iii)] for every $z \in Z$, any two points of $\pi ^{-1}(z)$ are in $m$-rc$(\mathscr{V}_1, \ldots , \mathscr{V}_k)$-relation. 
\end{itemize}
We call the morphism $\pi$ \textit{the rc$(\mathscr{V}_1, \ldots ,\mathscr{V}_k)$-fibration}. 
If $Z$ is a point, then we say that $X$ is rc$(\mathscr{V}_1, \ldots , \mathscr{V}_k)$-connected. 
\end{dfn}

\begin{dfn}
We denote by $N_1(X)$ the $\mathbb{R}$-vector space of 1-cycles with real coefficients modulo numerical equivalence. 
The Picard number $\rho_X$ is the dimension of the $\mathbb{R}$-vector space $N_1(X)$. 
For a subvariety $Y \subset X$, we denote by $N_1(Y, X)$ (resp.\ $NE(Y, X)$) the vector space (resp.\ cone) in $N_1(X)$ which is generated by numerical classes of curves contained in $Y$. 

For a family $\mathscr{V}$ of rational curves, we denote by  $[\mathscr{V}]$ the numerical class of a curve of $\mathscr{V}$, and by $D\cdot \mathscr{V}$ the intersection number of a divisor $D$ and a curve of $\mathscr{V}$
\end{dfn}

\begin{lem}[{see \cite[Lemmas 4.1 and 5.1]{ACO} and \cite[Lemmas 3.2]{Oc}}]\label{num}
Let $\mathscr{V}$ be a family of rational curves, $Y \subset X$ a closed subvariety, and $C$ a curve contained in ${\rm Locus}(\overline{\mathscr{V}};Y)$. 
Then we can write
$$[C] =a [C_Y] + \sum_{i=1}^k b_i [C_i],$$
where $C_Y$ is a curve contained in $Y$, each $C_i$ is an irreducible component of a cycle parametrized by $\overline{\mathscr{V}}_Y$, and $a,b_i \in \mathbb{Q}$. 
Furthermore, if $\mathscr{V}$ is unsplit, then we can write
$$[C] =a [C_Y] + b [\mathscr{V}],$$
where $C_Y$ is a curve contained in $Y$, $a \in \mathbb{Q}_{\ge 0}$, and $b \in \mathbb{Q}$. 
\end{lem}

\begin{lem}[{see \cite[Corollary 4.4]{ACO}}]\label{rcc}
Let $\mathscr{V}_1,$ $\ldots ,$ $\mathscr{V}_k$ be quasi-unsplit families of rational curves. 
Assume that $X$ is rc$(\mathscr{V}_1, \ldots , \mathscr{V}_k)$-connected. 
Then $N_1(X)$ is generated by $[\mathscr{V}_1],$ $\ldots ,$ $[\mathscr{V}_k]$. 
\end{lem}

\begin{lem}[{see \cite[IV Corollary 2.6]{Ko} and \cite[Lemma 5.4]{ACO}}]\label{dim}
Let $\mathscr{V}_1,$ $\ldots ,$ $\mathscr{V}_m$ be numerically independent families of rational curves. 
Assume that $\mathscr{V}_2,$ $\ldots ,$ $\mathscr{V}_m$ are unsplit. \\[3pt]
{\rm (1)} Let $x \in X$ be a point such that 
\begin{itemize}
\item ${\rm Locus}(\mathscr{V}_m,\ldots ,\mathscr{V}_1; x)$ is non-empty,
\item $\mathscr{V}_1$ is unsplit at $x$. 
\end{itemize}
Then 
$${\rm dim}\, {\rm Locus}(\mathscr{V}_m,\ldots ,\mathscr{V}_1; y)  \geq {\rm codim}_X {\rm Locus}(\mathscr{V}_1)+\sum_{i=1} ^m(-K_X \cdot \mathscr{V}_i)-m.$$ 
{\rm (2)} Assume that $\mathscr{V}_1$ is also unsplit. 
Let $Y \subset X$ be an irreducible closed subvariety such that
\begin{itemize}
\item ${\rm Locus}(\mathscr{V}_m,\ldots ,\mathscr{V}_1; Y)$ is non-empty,
\item $NE(Y,X) \cap (\mathbb{R}[\mathscr{V}_1]+ \cdots +\mathbb{R}[\mathscr{V}_m])=\{0\}$.
\end{itemize}
Then 
$${\rm dim}\, {\rm Locus}(\mathscr{V}_m,\ldots ,\mathscr{V}_1; Y)  \geq {\rm dim}\,Y+\sum_{i=1} ^m(-K_X \cdot \mathscr{V}_i)-m.$$ 
\end{lem}

\section{Fano manifolds}

Throughout this section, let $X$ be a Fano manifold (i.e., a smooth complex projective variety with ample anti-canonical line bundle $-K_X$). 

\begin{dfn}\label{index}
We denote by $r_X$ the greatest positive integer $r$ such
that $-K_X = rH$ for some line bundle $H$, which is called the \textit{index} of $X$. 
\end{dfn}

\begin{dfn}\label{pindex}
We denote
by $i_X$ the minimum of intersection numbers of $-K_X$ with rational curves on $X$, which is
called the \textit{pseudo-index} of $X$. 
Clearly $i_X \geq r_X$ holds. 
\end{dfn}

\begin{dfn}
Let $\pi : U \rightarrow Z$ be a proper surjective morphism mapping from an open dense subvariety $U \subset X$ to a quasi-projective variety $Z$. 
We say that a family $\mathscr{V}$ of rational curves is \textit{horizontal dominating} with respect to $\pi$ if ${\rm Locus}(\mathscr{V})$ dominates $Z$ and curves of $\mathscr{V}$ are not contracted by $\pi$. 
Since $X$ is a Fano manifold, $X$ admits such a family (\cite[Theorem 2.1]{KMM}). 
If a horizontal dominating family $\mathscr{V}$ has minimal anti-canonical degree among such families, then we call it a \textit{minimal horizontal dominating family} with respect to $\pi$. 
When $\pi={\rm id}_X$, it is called a \textit{minimal dominating family}. 

Remark that a minimal horizontal dominating family with respect to $\pi$ is unsplit on a general fiber of $\pi$, so it is locally unsplit. 
In particular, any Fano manifold $X$ admits a locally unsplit dominating family of rational curves. 
\end{dfn}

\begin{const}[{cf.\ \cite[Construction 1]{NO}}]\label{const1}
Suppose that $X$ is a Fano $n$-fold with $i_X \geq 2$. 
Let $\mathscr{F}_1$ be a locally unsplit dominating family of rational curves, and $\pi_1: X  \dashrightarrow Z_1$ the rc$(\mathscr{F}_1)$-fibration. 
If ${\rm dim}\,Z_1 >0$, then let $\mathscr{F}_2$ be a minimal horizontal dominating family with respect to $\pi_1$, and $\pi_2: X  \dashrightarrow Z_2$ the rc$(\mathscr{F}_1, \mathscr{F}_2)$-fibration, and so on. 
Since ${\rm dim}\,Z_i< {\rm dim}\,Z_{i-1}$, we have ${\rm dim}\,Z_k=0$ for some integer $k$, i.e., $X$ is rc$(\mathscr{F}_1, \ldots ,\mathscr{F}_k)$-connected. 
\end{const}

\begin{rmk}[{see \cite[Lemma 4]{NO}}]\label{const1rmk}
In Construction \ref{const1}, let $F_i$ be a general fiber of $\pi_i$. 
Since $\mathscr{F}_i$ is a minimal horizontal dominating family with respect to $\pi_{i-1}$, we know: 
\begin{itemize}
\item ${\rm Locus}(\mathscr{F}_i)$ intersects $F_{i-1}$, 
\item $\mathscr{F}_i$ is unsplit on $F_{i-1}$, 
\item $[\mathscr{F}_i]$ does not belong to $N_1(F_{i-1},X)$. 
\end{itemize}
By applying Lemmas \ref{num} and \ref{dim}, we have: 
$${\rm dim}\,Z_{i-1}-{\rm dim}\,Z_i \ge {\rm dim}\, {\rm Locus}(\mathscr{F}_i;y) \ge  {\rm codim}_X {\rm Locus}(\mathscr{F}_i)-K_X \cdot \mathscr{F}_i-1,$$
where $y$ is a general point of $ {\rm Locus}(\mathscr{F}_i)$. 
This implies
$$n \ge \sum_{i=1}^k \{ {\rm codim}_X {\rm Locus}(\mathscr{F}_i)-K_X \cdot \mathscr{F}_i-1\}.$$
We also have that $\mathscr{F}_1,$ $\ldots ,$ $\mathscr{F}_k$ are numerically independent. 
\end{rmk}

\begin{lem}\label{lem0}
Let $X$ be a Fano $n$-fold with $i_X \geq 2$. 
For any unsplit family $\mathscr{V}$ of rational curves, ${\rm dim}\,{\rm Locus}(\mathscr{V}) > \frac{n}{2}$ holds. 
\end{lem}

\begin{proof}
Let $d$ be the dimension of ${\rm Locus}(\mathscr{V})$. 
By Lemma \ref{dim}, for any point $y \in {\rm Locus}(\mathscr{V})$, we have: 
$$d \ge {\rm dim}\,{\rm Locus}(\mathscr{V};y) \ge (n-d) +2-1.$$
This implies $d > \frac{n}{2}$. 
\end{proof}

\begin{lem}\label{Ca2}
Let $X$ be a Fano manifold $X$ with $i_X \ge 2$ which admits no unsplit dominating families of rational curves. 
Then, for any prime divisor $D \subset X$, we have $N_1(D,X)=N_1(X)$. 
\end{lem}

\begin{proof}
By \cite[Theorem 1.2]{Ca2}, any Fano manifold with $i_X \ge 2$ satisfies either
\begin{itemize}
\item[(i)] $i_X=2$ and $X$ is a $\mathbb{P}^1$-bundle over a Fano manifold $Y$ with $i_Y \geq 2$, or
\item[(ii)] $N_1(D,X)=N_1(X)$ for any prime divisor $D \subset X$. 
\end{itemize}
If (i) holds, then the family of fibers is an unsplit dominating family of rational curves on $X$. 
Thus (ii) holds.  
\end{proof}

\section{Main results}

In this section, we consider Fano $6$-folds $X$ with $i_X \ge 2$ which admit no unsplit dominating families of rational curves. 
Note that the locus of any unsplit family of rational curves has dimension $4$ or $5$ by Lemma \ref{lem0}. 

\begin{prop}\label{prop1}
Let $X$ be a Fano $6$-fold with $i_X \ge 2$ which admits no unsplit dominating families of rational curves, and let $\mathscr{F}_1,$ $\ldots ,$ $\mathscr{F}_k$ be as in Construction \ref{const1}. 
Then we have one of the following: 
\begin{itemize}
\item[(a)] $k=1$,
\item[(b)] $k=2$ and $\mathscr{F}_2$ is unsplit, 
\item[(c)] $k=2$, $\mathscr{F}_2$ is non-unsplit, and $\rho_X=2$. 
\end{itemize}
Furthermore, in case {\rm (b)}, we have either
\begin{itemize}
\item[(b1)] ${\rm dim}\,{\rm Locus}(\mathscr{F}_2)=5$, and $N_1(X) = N_1(F_1,X)+ \mathbb{R}[\mathscr{F}_2]$ for a general fiber $F_1$ of  the rc$(\mathscr{F}_1)$-fibration $\pi_1$, or
\item[(b2)] ${\rm Locus}(\mathscr{F}_2,\mathscr{F}_1;x) \ne \emptyset$ for a general point $x \in X$. 
\end{itemize}
\end{prop}

\begin{proof}
Let $m$ be the number of non-unsplit families among  $\mathscr{F}_1,$ $\ldots ,$ $\mathscr{F}_k$. 
Since $\mathscr{F}_1$ is dominating (so non-unsplit), we know $m \ge 1$. 
If $\mathscr{F}_i$ is non-unsplit, then 
$${\rm codim}_X {\rm Locus}(\mathscr{F}_i)-K_X \cdot \mathscr{F}_i-1 \ge 2i_X -1\ge 3.$$
On the other hand, if $\mathscr{F}_i$ is unsplit, then 
$${\rm codim}_X {\rm Locus}(\mathscr{F}_i)-K_X \cdot \mathscr{F}_i-1 \ge 1+i_X -1\ge 2.$$ 
By Remark \ref{const1rmk}, we get:
$$6 \ge 3m+2(k-m) = 2k+m,$$
thus $(k,m)$ is equal to one of $(1,1)$, $(2,1)$, or $(2,2)$. 

First, we consider the case $(k,m)=(2,2)$ (i.e., $\mathscr{F}_2$ is non-unsplit). 
Let $x$ be a general point of $X$, and $F_1$ the fiber of $\pi_1$ passing through $x$. 
Since the equality of Remark \ref{const1rmk} holds, we get: 
$${\rm dim}\,{\rm Locus}(\mathscr{F}_1;x)={\rm dim}\,F_1=3,$$
so ${\rm Locus}(\mathscr{F}_1;x)=F_1$. 
We also get that $\mathscr{F}_2$ is dominating, and for a general point $y \in X$, 
$${\rm dim}\,{\rm Locus}(\mathscr{F}_2;y)={\rm dim}\,Z_1=3,$$
so ${\rm Locus}(\mathscr{F}_2;y)$ dominates $Z_1$. 
This yields $X={\rm Locus}(\overline{\mathscr{F}_2};F_1)$ because $y$ is a general point of $X$. 
Recall that $\mathscr{F}_2$ is unsplit on $F_1$. 
By Lemma \ref{num}, we have: 
$$N_1(X)=N_1(F_1,X)+ \mathbb{R}[\mathscr{F}_2]=N_1({\rm Locus}(\mathscr{F}_1;x),X)+ \mathbb{R}[\mathscr{F}_2]=\mathbb{R}[\mathscr{F}_1]+\mathbb{R}[\mathscr{F}_2].$$
Thus (c) holds. 

Next, we consider the case $(k,m)=(2,1)$ (i.e., $\mathscr{F}_2$ is unsplit). 
Let $x$ be a general point of $X$, and $F_1$ the fiber of $\pi_1$ passing through $x$. 
If $F_1 ={\rm Locus}(\mathscr{F}_1;x)$, then 
$${\rm Locus}(\mathscr{F}_2,\mathscr{F}_1;x)={\rm Locus}(\mathscr{F}_2;F_1)\ne \emptyset,$$
i.e., (b2) holds. 
So, we suppose ${\rm dim}\,F_1 > {\rm dim}\,{\rm Locus}(\mathscr{F}_1;x)$. 
Then, by Remark \ref{const1rmk}, 
$$4 \ge 5 - {\rm codim}_X {\rm Locus}(\mathscr{F}_2) \ge 6- {\rm dim}\,{\rm Locus}(\mathscr{F}_2;y) \ge {\rm dim}\,F_1 > {\rm dim}\,{\rm Locus}(\mathscr{F}_1;x) \ge 3,$$
where $y$ is a general point of ${\rm Locus}(\mathscr{F}_2)$. 
Hence, ${\rm dim}\,{\rm Locus}(\mathscr{F}_2)=5$ and 
$${\rm dim}\,{\rm Locus}(\mathscr{F}_2;y)={\rm dim}\,Z_1=2,$$
so ${\rm Locus}(\mathscr{F}_2;y)$ dominates $Z_1$.
This yields ${\rm Locus}(\mathscr{F}_2)={\rm Locus}(\mathscr{F}_2;F_1)$. 
By applying Lemmas \ref{Ca2} and \ref{num}, we obtain: 
$$N_1(X)=N_1({\rm Locus}(\mathscr{F}_2),X)=N_1(F_1,X)+ \mathbb{R}[\mathscr{F}_2].$$
Thus (b1) holds. 
\end{proof}

\begin{rmk}
As shown later (Proposition \ref{prop5}), we can prove $\rho_X \le 4$ in case (b1). 
\end{rmk}

\begin{cor}[{Theorem \ref{main}(1)}]
Let $X$ be a Fano $6$-fold with $i_X \ge 2$ which admits no unsplit dominating families of rational curves. 
Assume that $X$ admits a contraction of fiber type. Then $\rho_X \le 2$ holds. 
\end{cor}

\begin{proof}
Let $R$ be an extremal ray of fiber type. 
Then there exits a dominating family $\mathscr{F}_1$ of rational curves such that $[\mathscr{F}_1] \in R$. 
Suppose that $\mathscr{F}_1$ has minimal anti-canonical degree among such families. 
Then $\mathscr{F}_1$ is locally unsplit. 
Since $R=\mathbb{R}_{\ge 0} [\mathscr{F}_1]$ is an extremal ray, we also have that $\mathscr{F}_1$ is quasi-unsplit. 
Let $\mathscr{F}_2,$ $\ldots ,$ $\mathscr{F}_k$ be as in Construction \ref{const1}. 
According to Proposition \ref{prop1}, we may assume either (a) $k=1$ or (b) $k=2$ and $\mathscr{F}_2$ is unsplit. 
By applying Lemma \ref{rcc}, we obtain that $\rho_X=1$ in case (a), and that $\rho_X =2$ in case (b). 
\end{proof}

\begin{lem}\label{lem1}
Let $X$ be a Fano $6$-fold with $i_X \ge 2$ which admits no unsplit dominating families of rational curves. 
Let $\mathscr{H}_1, \mathscr{H}_2, \mathscr{H}_3$ be numerically independent unsplit families of rational curves. 
Set $H_i:={\rm Locus}(\mathscr{H}_i)$. 
\begin{enumerate}
\item If ${\rm dim}\,H_1=4$ and there exists a connected chain parametrized by $(\mathscr{H}_1, \mathscr{H}_2, \mathscr{H}_3)$ (see Definition \ref{chain}), then $\rho_X =3$. 
\item If ${\rm dim}\,H_1=5$, ${\rm dim}\,H_2=4$, $H_1 \cdot \mathscr{H}_2 >0$, and $H_2$ intersects $H_3$, then $\rho_X =3$.
 \item If ${\rm dim}\,H_1={\rm dim}\,H_2 =5$, $H_1 \cdot \mathscr{H}_2 >0$, and $H_2 \cdot \mathscr{H}_3 >0$, then $\rho_X \le 4$.
\end{enumerate}
\end{lem}

\begin{proof}
In case (1), let $\{ C_i\}_{i=1}^3$ be a connected chain parametrized by $(\mathscr{H}_1, \mathscr{H}_2, \mathscr{H}_3)$, and take a point $y \in C_1$. 
Since ${\rm Locus}(\mathscr{H}_3, \mathscr{H}_2, \mathscr{H}_1;y)$ is non-empty, Lemma \ref{dim} yields: 
$${\rm dim}\,{\rm Locus}(\mathscr{H}_3, \mathscr{H}_2, \mathscr{H}_1;y) \ge 2+2+2+2-3=5.$$
Thus, by applying Lemmas \ref{Ca2} and \ref{num}, we obtain that $N_1(X)$ is generated by $[\mathscr{H}_1]$, $[\mathscr{H}_2]$, $[\mathscr{H}_3]$. 

Next, we prove (2). 
If  $N_1(H_2,X)$ is contained in $\mathbb{R}[\mathscr{H}_1]+\mathbb{R}[\mathscr{H}_2]$ (in particular, $[\mathscr{H}_3] \not\in N_1(H_2,X)$), then Lemma \ref{dim} implies: 
$${\rm dim}\,{\rm Locus}(\mathscr{H}_3;H_2) \ge 4+2-1=5.$$
Hence, Lemmas \ref{Ca2} and \ref{num} give that $N_1(X)$ is generated by $[\mathscr{H}_1]$, $[\mathscr{H}_2]$, $[\mathscr{H}_3]$. 

If $N_1(H_2,X)$ is not contained in  $\mathbb{R}[\mathscr{H}_1]+\mathbb{R}[\mathscr{H}_2]$, then there exists an irreducible curve $C \subset H_2$ such that $[C]$ does not belong to $\mathbb{R}[\mathscr{H}_1]+\mathbb{R}[\mathscr{H}_2]$. 
Since $H_1 \cdot \mathscr{H}_2 >0$, for any point $y \in C$, we know that ${\rm Locus}(\mathscr{H}_1, \mathscr{H}_2;y)$ is non-empty, so
$${\rm dim}\,{\rm Locus}(\mathscr{H}_1,\mathscr{H}_2;y) \ge 2+2+2-2=4$$
by Lemma \ref{dim}. 
This implies: 
$${\rm dim}\,{\rm Locus}(\mathscr{H}_1,\mathscr{H}_2;C) = {\rm dim}\,\bigcup_{y \in C}{\rm Locus}(\mathscr{H}_1,\mathscr{H}_2;y)\ge 5.$$
Indeed, if this locus also has dimension $4$, then there exists a point $z$ of this locus such that ${\rm Locus}(\mathscr{H}_2,\mathscr{H}_1;z)$ contains $C$, so $[C]$ must belong to $\mathbb{R}[\mathscr{H}_1]+\mathbb{R}[\mathscr{H}_2]$ by Lemma \ref{num}. 
It follows from Lemmas \ref{Ca2} and \ref{num} that $N_1(X)$ is generated by $[\mathscr{H}_1]$, $[\mathscr{H}_2]$, $[C]$. 

Finally, we prove (3). 
We can exchange the role of  $\mathscr{H}_1$ and $\mathscr{H}_3$ in (1), so we may assume ${\rm dim}\,H_3=5$. 
Then $N_1(H_3,X)=N_1(X)$ by Lemma \ref{Ca2}. 
If $\rho_X \ge 4$, then there exists an irreducible curve $C \subset H_3$ such that $[C]$ does not belong to $\mathbb{R}[\mathscr{H}_1]+\mathbb{R}[\mathscr{H}_2]+\mathbb{R}[\mathscr{H}_3]$. 
Since $H_1 \cdot \mathscr{H}_2 >0$ and $H_2 \cdot \mathscr{H}_3 >0$, for any point $y \in C$, we have that ${\rm Locus}(\mathscr{H}_1, \mathscr{H}_2, \mathscr{H}_3;y)$ is non-empty, so
$${\rm dim}\,{\rm Locus}(\mathscr{H}_1,\mathscr{H}_2, \mathscr{H}_3;y) \ge 1+2+2+2-3=4.$$
This implies: 
$${\rm dim}\,{\rm Locus}(\mathscr{H}_1,\mathscr{H}_2,\mathscr{H}_3;C) = {\rm dim}\,\bigcup_{y \in C}{\rm Locus}(\mathscr{H}_1,\mathscr{H}_2,\mathscr{H}_3;y)\ge 5.$$
Therefore, $N_1(X)$ is generated by $[\mathscr{H}_1]$, $[\mathscr{H}_2]$, $[\mathscr{H}_3]$, $[C]$. 
\end{proof}

\begin{const}\label{const2}
Let $X$ be a Fano $n$-fold. 
Assume that there exists a locally unsplit dominating family $\mathscr{F}$ of rational curves on $X$ such that $-K_X\cdot \mathscr{F} < 3i_X$ (in particular, any reducible cycle parametrized by $\overline{\mathscr{F}}$ has just two irreducible components). 
Let $\pi: X  \dashrightarrow Z$ be the rc$(\mathscr{F})$-fibration, $x \in X$ a general point, $F$ the fiber of $\pi$ passing through $x$, and $\rho$ the dimension of the vector space $N_1(F,X)$. 
Note that 
$${\rm Locus}(\overline{\mathscr{F}};x)\subset {\rm Locus}_2(\overline{\mathscr{F}};x)\subset \cdots \subset {\rm Locus}_m(\overline{\mathscr{F}};x)=F$$
for some integer $m$. 
We take integers $1 \le m_1 \le \cdots \le m_{\rho-1} < m$ and pairs of families $(\mathscr{V}_1,\mathscr{W}_1),$ $\ldots ,$ $(\mathscr{V}_{\rho-1},\mathscr{W}_{\rho-1})$ as follows: 

First, if $\rho \ge 2$, then let $m_1$ be the greatest integer such that
$$N_1({\rm Locus}_{m_1}(\overline{\mathscr{F}};x),X) \subset \mathbb{R}[\mathscr{F}].$$
Since
$$N_1({\rm Locus}_{m_1+1}(\overline{\mathscr{F}};x),X) \not\subset \mathbb{R}[\mathscr{F}],$$
Lemma \ref{num} implies that there exists a reducible cycle $C_1+C_1'$ parametrized by $\overline{\mathscr{F}}$ such that
\begin{itemize}
\item $x$ and $C_1$ can be connected by the $m_1$-tuple $(\overline{\mathscr{F}},\ldots , \overline{\mathscr{F}})$ (see Definition \ref{chain}), 
\item neither $[C_1]$ nor $[C_1']$ belongs to $\mathbb{R}[\mathscr{F}]$. 
\end{itemize}
Let $\mathscr{V}_1$ (resp.\ $\mathscr{W}_1$) be a family of deformations of $C_1$ (resp.\ $C_1'$). 

Next, if $\rho \ge 3$, then let $m_2$ be the greatest integer such that
$$N_1({\rm Locus}_{m_2}(\overline{\mathscr{F}};x),X) \subset \mathbb{R}[\mathscr{F}]+\mathbb{R}[\mathscr{V}_1].$$
Since
$$N_1({\rm Locus}_{m_2+1}(\overline{\mathscr{F}};x),X) \not\subset \mathbb{R}[\mathscr{F}]+\mathbb{R}[\mathscr{V}_1],$$
Lemma \ref{num} gives a reducible cycle $C_2+C_2'$ parametrized by $\overline{\mathscr{F}}$ such that
\begin{itemize}
\item $x$ and $C_2$ can be connected by the $m_2$-tuple $(\overline{\mathscr{F}},\ldots , \overline{\mathscr{F}})$, 
\item neither $[C_2]$ nor $[C_2']$ belongs to $\mathbb{R}[\mathscr{F}]+\mathbb{R}[\mathscr{V}_1]$. 
\end{itemize}
Let $\mathscr{V}_2$ (resp.\ $\mathscr{W}_2$) be a family of deformations of $C_2$ (resp.\ $C_2'$), and so on. 

Set $V_i:={\rm Locus}(\mathscr{V}_i)$ and $W_i:={\rm Locus}(\mathscr{W}_i)$. 
By construction, we know: 
\begin{enumerate}
\item $V_i \cap W_i \ne \emptyset$, 
\item $\mathscr{V}_i$ and $\mathscr{W}_i$ are unsplit, 
\item $[\mathscr{V}_i]+[\mathscr{W}_i]=[\mathscr{F]}$, 
\item $\mathscr{F},$ $\mathscr{V}_1,$ $\ldots ,$ $\mathscr{V}_{\rho-1}$ are numerically independent, 
\item $N_1(F,X)$ is generated by $[\mathscr{F}],$ $[\mathscr{V}_1],$ $\ldots ,$ $[\mathscr{V}_{\rho-1}]$. 
\end{enumerate}
\end{const}

\begin{prop}\label{prop2}
Let $X$ be a Fano $6$-fold with $\rho_X \ge 3$ and $i_X \ge 2$ which admits no unsplit dominating families of rational curves. 
Let $\mathscr{F}$ be a locally unsplit dominating family of rational curves. 
Let $(\mathscr{V}_1,\mathscr{W}_1),$ $\ldots ,$ $(\mathscr{V}_{\rho-1},\mathscr{W}_{\rho-1})$ be as in Construction \ref{const2}. 
Assume that $\rho \ge 3$ and ${\rm dim}\,V_i=5$ for some $i$. 
Then $\rho_X \le 4$ holds. 
\end{prop}

\begin{rmk}\label{prop2rmk}
Under the assumption of Proposition \ref{prop2}, $-K_X\cdot \mathscr{F} \le 5 <3i_X$ holds. 
Indeed, if $-K_X\cdot \mathscr{F} \ge 6$, then Lemma \ref{dim} yields: 
$${\rm dim}\,{\rm Locus}(\mathscr{F};x) \ge 6-1 =5,$$
so $N_1(X)$ must be generated by only $[\mathscr{F}]$ according to Lemmas \ref{Ca2} and \ref{num}. 
\end{rmk}

\begin{proof}[Proof of Proposition \ref{prop2}]
We use the notation of Construction \ref{const2}. 
Note that $D\cdot \mathscr{F} \ge 0$ for any effective divisor $D$ on $X$ because $\mathscr{F}$ is a dominating family. 

\begin{case1}
${\rm dim}\,V_1=5$. 
\end{case1}

By construction, $x$ and $V_1$ can be connected by the $m_1$-tuple $(\overline{\mathscr{F}},\ldots , \overline{\mathscr{F}})$. 
Since $x$ is a general point of $X$, we may assume $x \not\in V_1$. 
Hence, we get a rational curve $C \subset {\rm Locus}_{m_1}(\overline{\mathscr{F}};x)$ such that $V_1\cdot C >0$. 
This implies $V_1\cdot \mathscr{F}>0$. 

\begin{case12}
${\rm dim}\,V_2=4$. 
\end{case12}

Since $[\mathscr{F}]=[\mathscr{V}_2]+[\mathscr{W}_2]$, we obtain either 
\begin{itemize}
\item[(i)] $V_1\cdot \mathscr{V}_2>0$ or
\item[(ii)]  $V_1\cdot \mathscr{W}_2>0$. 
\end{itemize}
We can apply Lemma \ref{lem1}(2) for $(\mathscr{H}_1, \mathscr{H}_2, \mathscr{H}_3)=(\mathscr{V}_1, \mathscr{V}_2, \mathscr{W}_2)$ in case (i), and Lemma \ref{lem1}(1) for $(\mathscr{H}_1, \mathscr{H}_2, \mathscr{H}_3)=(\mathscr{V}_2, \mathscr{W}_2, \mathscr{V}_1)$ in case (ii). 
Thus both cases gives $\rho_X=3$. 

\begin{case12}
${\rm dim}\,V_2=5$. 
\end{case12}

We know that $x$ and $V_2$ can be connected by the $m_2$-tuple $(\overline{\mathscr{F}},\ldots , \overline{\mathscr{F}})$. 
Since we may assume $x \not\in V_2$, we get a rational curve $C \subset {\rm Locus}_{m_2}(\overline{\mathscr{F}};x)$ such that $V_2\cdot C >0$. 
This implies: 
$$V_2\cdot (a \mathscr{F} + b \mathscr{V}_1) >0$$
for some numbers $a$ and $b$. 
Since 
$$V_2\cdot \mathscr{V}_1+V_2\cdot \mathscr{W}_1 = V_2\cdot \mathscr{F} \ge 0,$$
we obtain either $V_2\cdot \mathscr{V}_1>0$ or $V_2\cdot \mathscr{W}_1>0$. 

\begin{case13}\label{case1}
$V_2\cdot \mathscr{V}_1>0$.
\end{case13}

Note that $\mathscr{F}$ is unsplit at $x$ because $\mathscr{F}$ is locally unsplit. 
Since $V_1\cdot \mathscr{F}>0$ and $V_2\cdot \mathscr{V}_1>0$, we have that ${\rm Locus}(\mathscr{V}_2, \mathscr{V}_1, \mathscr{F};x)$ is non-empty, and Lemma \ref{dim} yields: 
$${\rm dim}\,{\rm Locus}(\mathscr{V}_2, \mathscr{V}_1, \mathscr{F};x) \ge 4+2+2-3=5.$$
It follows from Lemmas \ref{Ca2} and \ref{num} that $N_1(X)$ is generated by $[\mathscr{F}]$, $[\mathscr{V}_1]$, $[\mathscr{V}_2]$. 

\begin{case13}
$V_2\cdot \mathscr{W}_1>0$ and ${\rm dim}\,W_1=4$. 
\end{case13}

We obtain $\rho_X=3$ by applying Lemma \ref{lem1}(2) for $(\mathscr{H}_1, \mathscr{H}_2, \mathscr{H}_3)=(\mathscr{V}_2, \mathscr{W}_1, \mathscr{V}_1)$. 

\begin{case13}
$V_2\cdot \mathscr{W}_1>0$ and ${\rm dim}\,W_1=5$. 
\end{case13}

By construction, $x$ and $W_1$ can be connected by the $(m_1+1)$-tuple $(\overline{\mathscr{F}},\ldots , \overline{\mathscr{F}},\mathscr{V}_1)$. 
Since we may assume $x \not\in W_1$, we get a rational curve $C$ such that $W_1\cdot C >0$ and either $[C] \in \mathbb{R}[\mathscr{F}]$ or $C \in \mathscr{V}_1$. 
This yields either 
\begin{itemize}
\item[(i)] $W_1 \cdot \mathscr{F}>0$ or
\item[(ii)] $W_1 \cdot \mathscr{V}_1>0$. 
\end{itemize}
If (i) holds, then we can replace $\mathscr{V}_1$ with $\mathscr{W}_1$ in the proof of Case \ref{case1}, so we obtain that $N_1(X)$ is generated by $[\mathscr{F}]$, $[\mathscr{W}_1]$, $[\mathscr{V}_2]$. 
If (ii) holds, then we have $\rho_X \le 4$ by applying Lemma \ref{lem1}(3) for $(\mathscr{H}_1, \mathscr{H}_2, \mathscr{H}_3)=(\mathscr{V}_2, \mathscr{W}_1, \mathscr{V}_1)$. 

\begin{case1}
${\rm dim}\,V_i=5$ for some $i \ge 2$. 
\end{case1}

We may assume ${\rm dim}\,V_1=\cdots ={\rm dim}\,V_{i-1}=4$. 
By construction, $x$ and $V_i$ can be connected by the $m_i$-tuple $(\overline{\mathscr{F}},\ldots , \overline{\mathscr{F}})$. 
Since we may assume $x \not\in V_i$, we get a rational curve $C \subset {\rm Locus}_{m_i}(\overline{\mathscr{F}};x)$ such that $V_i\cdot C >0$. 
This implies: 
$$V_i\cdot (a \mathscr{F} + \sum_{j=1}^{i-1} b_j\mathscr{V}_j) >0$$
for some numbers $a$ and $b_j$. 
Since 
$$V_i\cdot \mathscr{V}_j+V_i\cdot \mathscr{W}_j = V_i\cdot \mathscr{F} \ge 0,$$
we obtain either 
\begin{itemize}
\item[(i)] $V_i\cdot \mathscr{V}_j>0$ or
\item[(ii)] $V_i\cdot \mathscr{W}_j>0$
\end{itemize}
for some $j \le i-1$. 
Notice that ${\rm dim}\,V_j=4$. 
We can use Lemma \ref{lem1}(2) for $(\mathscr{H}_1, \mathscr{H}_2, \mathscr{H}_3)=(\mathscr{V}_i, \mathscr{V}_j, \mathscr{W}_j)$ in case (i), and Lemma \ref{lem1}(1) for $(\mathscr{H}_1, \mathscr{H}_2, \mathscr{H}_3)=(\mathscr{V}_j, \mathscr{W}_j, \mathscr{V}_i)$ in case (ii) respectively. 
We thus obtain $\rho_X=3$. 
\end{proof}

\begin{prop}\label{prop5}
Let $X$ be a Fano $6$-fold and $i_X \ge 2$ which admits no unsplit dominating families of rational curves. 
Let $\mathscr{F}$ be a locally unsplit dominating family of rational curves, $\mathscr{G}$ an unsplit family of rational curves. 
Assume that
\begin{itemize}
\item $X$ is rc$(\mathscr{F},\mathscr{G})$-connected, 
\item ${\rm dim}\,{\rm Locus}(\mathscr{G})=5$, 
\item for a general fiber $F$ of the rc$(\mathscr{F})$-fibration, $[\mathscr{G}] \not\in N_1(F,X)$ and $N_1(X)=N_1(F,X)+\mathbb{R}[\mathscr{G}]$
\end{itemize}
Then $\rho_X \le 4$ holds. 
\end{prop}

\begin{proof}
Assume by contradiction that $\rho_X \ge 5$. 
We consider Construction \ref{const2}. 
By assumption, $\rho={\rm dim}\,N_1(F,X)=\rho_X-1 \ge 4$, so Proposition \ref{prop2} yields that ${\rm dim}\,V_i=4$ for every $1 \le i \le \rho-1$. 

Set $G:={\rm Locus}(\mathscr{G})$. 
Since $X$ is rc$(\mathscr{F},\mathscr{G})$-connected, $F={\rm Locus}_m(\overline{\mathscr{F}};x)$ intersects $G$. 
Since $x$ is a general point of $X$, we may assume $x \not\in G$. 
This implies that there exists a rational curve $C \subset F$ such that $G\cdot C>0$. 
Then $[C]$ belongs to $N_1(F,X)=\mathbb{R}[\mathscr{F}]+\mathbb{R}[\mathscr{V}_1]+\cdots +\mathbb{R}[\mathscr{V}_{\rho-1}]$, so we have either 
\begin{itemize}
\item[(i)] $G\cdot \mathscr{V}_i >0$ or
\item[(ii)]  $G\cdot \mathscr{W}_i >0$
\end{itemize}
for some $1 \le i \le \rho-1$. 
However, we can apply Lemma \ref{lem1}(2) for $(\mathscr{H}_1, \mathscr{H}_2, \mathscr{H}_3)=(\mathscr{G}, \mathscr{V}_i, \mathscr{W}_i)$ in case (i), and Lemma \ref{lem1}(1) for $(\mathscr{H}_1, \mathscr{H}_2, \mathscr{H}_3)=(\mathscr{V}_i, \mathscr{W}_i, \mathscr{G})$ in case (ii). 
Thus both cases give $\rho_X=3$, a contradiction. 
\end{proof}

According to Propositions \ref{prop1} and \ref{prop5}, it is sufficient to consider the cases (a) and (b2) in order to prove Theorem \ref{main}(2). 
First, we consider the case (a). 

\begin{prop}\label{prop3}
Let $X$ be a Fano $6$-fold with $\rho_X \ge 5$ and $i_X \ge 2$ which admits no unsplit dominating families of rational curves. 
Let $\mathscr{F}$ be a locally unsplit dominating family of rational curves. 
Let $(\mathscr{V}_1,\mathscr{W}_1),$ $\ldots ,$ $(\mathscr{V}_{\rho-1},\mathscr{W}_{\rho-1})$ be as in Construction \ref{const2}. 
Assume that $X$ is rc$(\mathscr{F})$-connected, and also that $X$ admits a divisorial extremal ray $R$. 
Then $R=\mathbb{R}_{\ge 0}[\mathscr{W}_i]$ for some $1 \le i \le \rho-1$. 
\end{prop}

\begin{proof}
We use the notation of Construction \ref{const2}. 
Notice that $F=X$ and $\rho=\rho_X$. 
By Proposition \ref{prop2}, we have ${\rm dim}\,V_i=4$ for every $1 \le i \le \rho-1$. 
We denote by $E_R$ the exceptional divisor with respect to $R$. 

\begin{claim}\label{claim1}
$E_R\cdot \mathscr{F}=0$. 
\end{claim}

\begin{proof}[Proof of Claim \ref{claim1}]
We assume by contradiction that $E_R\cdot \mathscr{F}>0$. 
Since $R$ is divisorial and $\mathscr{F}$ is dominating, we have: 
$$R \not\subset \mathbb{R}[\mathscr{F}] = (\mathbb{R}[\mathscr{F}] +\mathbb{R}[\mathscr{V}_1]) \cap(\mathbb{R}[\mathscr{F}] +\mathbb{R}[\mathscr{V}_2]),$$
so $R \not\subset (\mathbb{R}[\mathscr{F}] +\mathbb{R}[\mathscr{V}_i])$ for some $i$ ($=1,2$). 
This means that $R$, $[\mathscr{V}_i]$, $[\mathscr{W}_i]$ are independent. 
Since $[\mathscr{F}]=[\mathscr{V}_i]+[\mathscr{W}_i]$, we obtain either $E_R\cdot \mathscr{V}_i>0$ or $E_R\cdot \mathscr{W}_i>0$. 

\begin{case2}\label{case2}
$E_R\cdot \mathscr{W}_i>0$. 
\end{case2}

By Construction \ref{const2}, there exists a connected cycle $C_i+C_i'$ such that $C_i \in \mathscr{V}_i$ and $C_i' \in \mathscr{W}_i$. 
Since $E_R\cdot \mathscr{W}_i>0$, we have $E_R \cap C_i' \ne \emptyset$. 
So, there exists a family $\mathscr{H}$ of rational curves such that $[\mathscr{H}] \in R$ and ${\rm Locus}(\mathscr{H}) \cap C_i' \ne \emptyset$. 
Suppose that $\mathscr{H}$ has minimal anti-canonical degree among such families. 
Let $y$ be a point of ${\rm Locus}(\mathscr{H}) \cap C_i'$, then we see that $\mathscr{H}$ is unsplit at $y$. 

If $\mathscr{H}$ is non-unsplit, then Lemma \ref{dim} yields: 
$${\rm dim}\,{\rm Locus}(\mathscr{W}_i, \mathscr{H}; y) \ge 1+4+2-2=5,$$
so Lemmas \ref{Ca2} and \ref{num} imply that $N_1(X)$ is generated by $[\mathscr{H}]$ and $[\mathscr{W}_i]$, that is a contradiction. 
Thus $\mathscr{H}$ is unsplit. 

However, $\rho_X=3$ follows from Lemma \ref{lem1}(1) for $(\mathscr{H}_1, \mathscr{H}_2, \mathscr{H}_3)=(\mathscr{V}_i, \mathscr{W}_i, \mathscr{H})$, that is also a contradiction. 

\begin{case2}
$E_R\cdot \mathscr{V}_i>0$. 
\end{case2}

By replacing $\mathscr{W}_i$ with $\mathscr{V}_i$ in Case \ref{case2}, we get an unsplit family $\mathscr{H}$ of rational curves such that $[\mathscr{H}] \in R$ and there exists a connected chain parametrized by $(\mathscr{W}_i, \mathscr{V}_i, \mathscr{H})$. 

\begin{case22}
${\rm dim}\,W_i=4$. 
\end{case22}

In this case, $\rho_X$ must be $3$ by Lemma \ref{lem1}(1) for $(\mathscr{H}_1, \mathscr{H}_2, \mathscr{H}_3)=(\mathscr{W}_i, \mathscr{V}_i, \mathscr{H})$, a contradiction. 

\begin{case22}
${\rm dim}\,W_i=5$. 
\end{case22}

By construction, $x$ and $W_i$ can be connected by the $(m_i+1)$-tuple $(\overline{\mathscr{F}},\ldots , \overline{\mathscr{F}},\mathscr{V}_i)$. 
Since we may assume $x \not\in W_i$, we get a rational curve $C$ such that $W_i \cdot C >0$ and either $[C]$ belongs to $\mathbb{R}[\mathscr{F}]+\mathbb{R}[\mathscr{V}_1]+\cdots +\mathbb{R}[\mathscr{V}_{i-1}]$ or $C \in \mathscr{V}_i$. 
Since 
$$W_i\cdot \mathscr{V}_j+W_i\cdot \mathscr{W}_j = W_i\cdot \mathscr{F} \ge 0,$$
we obtain one of 
\begin{itemize}
\item[(i)] $W_i\cdot \mathscr{V}_j>0$ for some $j \ne i$, 
\item[(ii)]  $W_i\cdot \mathscr{W}_j>0$ for some $j \ne i$, 
\item[(iii)] $W_i\cdot \mathscr{V}_i>0$. 
\end{itemize}
However, we can apply Lemma \ref{lem1}(2) for $(\mathscr{H}_1, \mathscr{H}_2, \mathscr{H}_3)=(\mathscr{W}_i, \mathscr{V}_j, \mathscr{W}_j)$ in case (i), Lemma \ref{lem1}(1) for $(\mathscr{H}_1, \mathscr{H}_2, \mathscr{H}_3)=(\mathscr{V}_j, \mathscr{W}_j, \mathscr{W}_i)$ in case (ii), and Lemma \ref{lem1}(2) for $(\mathscr{H}_1, \mathscr{H}_2, \mathscr{H}_3)=(\mathscr{W}_i, \mathscr{V}_i, \mathscr{H})$ in case (iii) respectively. 
Therefore $\rho_X$ must be $3$, a contradiction. 
\end{proof}

Since $E_R$ is an effective divisor, there exists a curve $C$ such that $E_R\cdot C >0$. 
Recall that $N_1(X)$ is generated by $[\mathscr{F}],$ $[\mathscr{V}_1],$ $\ldots ,$ $[\mathscr{V}_{\rho-1}]$ and $E_R\cdot \mathscr{F}=0$, so $E_R\cdot \mathscr{V}_i \ne 0$ for some $1 \le i \le \rho-1$. 
From now on, let $i$ be the smallest index such that $E_R\cdot \mathscr{V}_i \ne 0$. 
Remark that $R$, $[\mathscr{V}_i]$, $[\mathscr{W}_i]$ are possibly not independent. 

\begin{claim}\label{claim2}
$E_R\cdot \mathscr{V}_i > 0$ and $E_R\cdot \mathscr{W}_i < 0$. 
\end{claim}

\begin{proof}[Proof of Claim \ref{claim2}]
We assume by contradiction that $E_R\cdot \mathscr{V}_i < 0$. 
Then we have $V_i \subset E_R$. 
So, $x$ and $E_R$ can be connected by the $m_i$-tuple $(\overline{\mathscr{F}},\ldots , \overline{\mathscr{F}})$. 
Since we may assume $x \not\in E_R$, we get a rational curve $C \subset {\rm Locus}_{m_i}(\overline{\mathscr{F}};x)$ such that $E_R\cdot C >0$. 
This implies: 
$$E_R\cdot (a \mathscr{F} + \sum_{j=1}^{i-1} b_j\mathscr{V}_j) >0$$
for some numbers $a$ and $b_j$. 
This gives a contradiction because $E_R\cdot \mathscr{F}=E_R\cdot \mathscr{V}_1 = \dots = E_R\cdot \mathscr{V}_{i-1}=0$. 
Consequently, we obtain $E_R\cdot \mathscr{V}_i \ (=-E_R\cdot \mathscr{W}_i) > 0$. 
\end{proof}

\begin{claim}\label{claim3}
$[\mathscr{W}_i] \in R$. 
\end{claim}

\begin{proof}[Proof of Claim \ref{claim3}]
We assume by contradiction that $[\mathscr{W}_i] \not\in R$. 
Note that $W_i \subset E_R$ by Claim \ref{claim2}. 
So, there exists a family $\mathscr{H}$ of rational curves such that $[\mathscr{H}] \in R$ and $W_i \subset {\rm Locus}(\mathscr{H})$. 
Suppose that $\mathscr{H}$ has minimal anti-canonical degree among such families. 
This implies that $\mathscr{H}$ is unsplit at some point of $W_i$. 
Since $\mathscr{H}$ and $\mathscr{W}_i$ are numerically independent, we see that $\mathscr{H}$ is unsplit as in Case \ref{case2} of Claim \ref{claim1}. 

By Construction \ref{const2}, there exists a point $y$ such that ${\rm Locus}(\mathscr{W}_i, \mathscr{V}_i ;y)$ is non-empty. 
Lemma \ref{dim} gives an irreducible component $Y$ of ${\rm Locus}(\mathscr{W}_i, \mathscr{V}_i ;y)$ with
$${\rm dim}\,Y \ge 2+2+2-2=4.$$
Then ${\rm Locus}(\mathscr{H};Y)$ is non-empty because $W_i \subset {\rm Locus}(\mathscr{H})$. 
According to Lemma \ref{num}, if there exists a curve $C \subset Y$ such that $C \in \mathbb{R}[\mathscr{H}]$, then we can write: 
$$[C]=a[\mathscr{V}_i]+b[\mathscr{W}_i]=c[\mathscr{H}]\ (\in R)$$
for some $a \in \mathbb{Q}_{\ge 0}$, $b \in \mathbb{Q}$, and $c \in \mathbb{Q}_{> 0}$. 
This implies $b>0$ because  $E_R\cdot \mathscr{V}_i > 0$, $E_R\cdot \mathscr{W}_i < 0$, and $E_R\cdot \mathscr{H} < 0$, so $[\mathscr{W}_i]$ must belong to the extremal ray $R$. 
Thus we have $NE(Y,X) \cap \mathbb{R}[\mathscr{H}] = \{0\}$. 
Hence, Lemma \ref{dim} yields: 
$${\rm dim}\,{\rm Locus}(\mathscr{H};Y) \ge 4+2-1=5,$$
so $N_1(X)$ must be generated by $[\mathscr{V}_i ], [\mathscr{W}_i ], [\mathscr{H}]$ according to Lemmas \ref{Ca2} and \ref{num}, a contradiction. 
\end{proof}

Consequently, we obtain the conclusion of Proposition \ref{prop3}. 
\end{proof}

\begin{cor}[{Theorem \ref{main}(2) for the case (a)}]
Let $X$ be a Fano $6$-fold with $i_X \ge 2$ which admits no unsplit dominating families of rational curves. 
Assume that $X$ is rc$(\mathscr{F})$-connected for some locally unsplit dominating family $\mathscr{F}$, and also that every extremal ray of $X$ is divisorial. 
Then $\rho_X \le 4$ holds. 
\end{cor}

\begin{proof}
If $\rho_X \ge 5$, then Proposition \ref{prop3} yields that $N_1(X)$ is generated by $[\mathscr{W}_1],$ $\ldots ,$ $[\mathscr{W}_{\rho-1}]$, that is a contradiction. 
\end{proof}

Finally, we consider the case (b2). 

\begin{ass}\label{Ass}
$X$ is a Fano $6$-fold with $i_X \ge 2$ which admits no unsplit dominating families of rational curves. 
$\mathscr{F}$ is a locally unsplit dominating family of rational curves, and $\mathscr{G}$ is an unsplit family of rational curves such that
\begin{itemize}
\item $\mathscr{F}$ and $\mathscr{G}$ are numerically independent, 
\item $X$ is rc$(\mathscr{F},\mathscr{G})$-connected, 
\item ${\rm Locus}(\mathscr{G},\mathscr{F};x) \ne \emptyset$ for a general point $x \in X$. 
\end{itemize}
Set $G:={\rm Locus}(\mathscr{G})$. 
Remark that the last condition means ${\rm Locus}(\overline{\mathscr{F}};G)=X$. 
\end{ass}

We consider the following construction, which is an analogue of Construction \ref{const2}. 

\begin{const}\label{const3}
Let $X$, $\mathscr{F}$, and $\mathscr{G}$ be as in Assumption \ref{Ass}. 
Note that $-K_X\cdot \mathscr{F} \le 5 <3i_X$ as in Remark \ref{prop2rmk}. 
Let $x$ a general point of $X$. 
Since $X$ is rc$(\mathscr{F},\mathscr{G})$-connected, we can write
$$X={\rm Locus}(\mathscr{H}_m, \ldots , \mathscr{H}_1;x),$$
where each $\mathscr{H}_i$ is either $\overline{\mathscr{F}}$ or $\mathscr{G}$. 
Set
$$U_i:=\bigcup _{j=1}^i{\rm Locus}(\mathscr{H}_j, \ldots , \mathscr{H}_1;x)$$
for $1 \le i \le m$, and $\rho :=\rho_X$. 
We take integers $1 \le m_1 \le \cdots \le m_{\rho-2} < m$ and pairs of families $(\mathscr{V}_1,\mathscr{W}_1),$ $\ldots ,$ $(\mathscr{V}_{\rho-2},\mathscr{W}_{\rho-2})$ as follows: 

First, if $\rho \ge 3$, then let $m_1$ be the greatest integer such that
$$N_1(U_{m_1},X) \subset \mathbb{R}[\mathscr{F}]+\mathbb{R}[\mathscr{G}].$$
Since
$$N_1(U_{m_1+1},X) \not\subset \mathbb{R}[\mathscr{F}]+\mathbb{R}[\mathscr{G}],$$
Lemma \ref{num} gives a reducible cycle $C_1+C_1'$ parametrized by $\overline{\mathscr{F}}$ ($=\mathscr{H}_{m_1+1}$) such that
\begin{itemize}
\item $x$ and $C_1$ can be connected by $(\mathscr{H}_1, \ldots , \mathscr{H}_{m_1})$, 
\item neither $[C_1]$ nor $[C_1']$ belongs to $\mathbb{R}[\mathscr{F}]+\mathbb{R}[\mathscr{G}]$. 
\end{itemize}
Let $\mathscr{V}_1$ (resp.\ $\mathscr{W}_1$) be a family of deformations of $C_1$ (resp.\ $C_1'$). 

Next, if $\rho \ge 4$, then let $m_2$ be the greatest integer such that
$$N_1(U_{m_2},X) \subset \mathbb{R}[\mathscr{F}]+\mathbb{R}[\mathscr{G}]+\mathbb{R}[\mathscr{V}_1].$$
Since
$$N_1(U_{m_2+1},X) \not\subset \mathbb{R}[\mathscr{F}]+\mathbb{R}[\mathscr{G}]+\mathbb{R}[\mathscr{V}_1],$$
Lemma \ref{num} gives a reducible cycle $C_2+C_2'$ parametrized by $\overline{\mathscr{F}}$ ($=\mathscr{H}_{m_2+1}$) such that
\begin{itemize}
\item $x$ and $C_2$ can be connected by $(\mathscr{H}_1, \ldots , \mathscr{H}_{m_2})$, 
\item neither $[C_2]$ nor $[C_2']$ belongs to $\mathbb{R}[\mathscr{F}]+\mathbb{R}[\mathscr{G}]+\mathbb{R}[\mathscr{V}_1]$. 
\end{itemize}
Let $\mathscr{V}_2$ (resp.\ $\mathscr{W}_2$) be a family of deformations of $C_2$ (resp.\ $C_2'$), and so on. 

Set $V_i:={\rm Locus}(\mathscr{V}_i)$ and $W_i:={\rm Locus}(\mathscr{W}_i)$. 
By construction, we know: 
\begin{enumerate}
\item $V_i \cap W_i \ne \emptyset$, 
\item $\mathscr{V}_i$ and $\mathscr{W}_i$ are unsplit, 
\item $[\mathscr{V}_i]+[\mathscr{W}_i]=[\mathscr{F]}$, 
\item $\mathscr{F},$ $\mathscr{G},$ $\mathscr{V}_1,$ $\ldots ,$ $\mathscr{V}_{\rho-2}$ are numerically independent, 
\item $N_1(X)$ is generated by $[\mathscr{F}], [\mathscr{G}], [\mathscr{V}_1], \ldots , [\mathscr{V}_{\rho-2}]$. 
\end{enumerate}
\end{const}

\begin{prop}\label{prop4}
Let $X$, $\mathscr{F}$, and $\mathscr{G}$ be as in Assumption \ref{Ass}, and let $(\mathscr{V}_1,\mathscr{W}_1),$ $\ldots,$ $(\mathscr{V}_{\rho-2},\mathscr{W}_{\rho-2})$ be as in Construction \ref{const3}. 
Assume that $\rho_X \ge 5$ and $X$ admits a divisorial extremal ray $R$. 
Then $R=\mathbb{R}_{\ge 0}[\mathscr{W}_i]$ for some $1 \le i \le \rho-2$. 
\end{prop}

\begin{proof}

\setcounter{claim}{0}

\begin{claim}\label{claim4}
If there exists an unsplit family $\mathscr{H}$ of rational curves such that
\begin{itemize}
\item $H:={\rm Locus}(\mathscr{H})$ has dimension $5$, 
\item $\mathscr{F}$, $\mathscr{G}$, $\mathscr{H}$ are numerically independent. 
\end{itemize}
Then $H \cdot \mathscr{G}=0$ holds. 
\end{claim}

\begin{proof}[Proof of Claim \ref{claim4}]
Let $x$ be a general point of $X$. 
If $H \cdot \mathscr{G} \ne 0$, then ${\rm Locus}(\mathscr{H}, \mathscr{G}, \mathscr{F};x)$ is non-empty, so Lemma \ref{dim} gives: 
$${\rm dim}\,{\rm Locus}(\mathscr{H}, \mathscr{G}, \mathscr{F};x) \ge 4+2+2-3=5.$$
It follows from Lemmas \ref{Ca2} and \ref{num} that $N_1(X)$ is generated by $[\mathscr{F}]$, $[\mathscr{G}]$, $[\mathscr{H}]$, a contradiction. 
\end{proof}

By Claim \ref{claim4}, if $V_j$ (resp. $W_j$) has dimension $5$ for some $j$, then $V_j\cdot \mathscr{G} =0$ (resp.\ $W_j\cdot \mathscr{G} =0$). 
This allows us to prove that 
\begin{itemize}
\item ${\rm dim}V_i =4$ for every $1 \le i \le \rho-2$, 
\item $E_R\cdot \mathscr{F}=0$, 
\item If $E_R\cdot \mathscr{G}=0$, then $R=\mathbb{R}_{\ge 0}[\mathscr{W}_i]$ for some $1 \le i \le \rho-2$, 
\end{itemize}
as in the proof of Propositions \ref{prop2} and \ref{prop3} by replacing $(\overline{\mathscr{F}},\overline{\mathscr{F}},\ldots)$ with $(\mathscr{H}_1, \mathscr{H}_2, \ldots)$. 
So we only need to show that $E_R\cdot \mathscr{G} = 0$. 

\begin{claim}\label{claim5}
${\rm dim}\,G=4$. 
\end{claim}

\begin{proof}[Proof of Claim \ref{claim5}]
We assume by contradiction that ${\rm dim}\,G=5$. 
Let $x$ be a general point of $X$. 
The assumption ${\rm Locus}(\mathscr{G}, \mathscr{F};x) \ne \emptyset$ means that $x$ and $G$ can be connected by a curve parametrized by $\mathscr{F}$. 
Since $x \not\in G$, we have $G\cdot \mathscr{F} >0$. 
This implies either 
\begin{itemize}
\item[(i)] $G\cdot \mathscr{V}_1 >0$ or
\item[(ii)] $G\cdot \mathscr{W}_1 >0$. 
\end{itemize}
We can apply Lemma \ref{lem1}(2) for $(\mathscr{H}_1, \mathscr{H}_2, \mathscr{H}_3)=(\mathscr{G}, \mathscr{V}_1, \mathscr{W}_1)$ in case (i), and Lemma \ref{lem1}(1) for $(\mathscr{H}_1, \mathscr{H}_2, \mathscr{H}_3)=(\mathscr{V}_1, \mathscr{W}_1, \mathscr{G})$ in case (ii) respectively. 
Therefore $\rho_X$ must be $3$, a contradiction. 
\end{proof}

\begin{claim}\label{claim7}
$N_1(G,X)$ is contained in $\mathbb{R}[\mathscr{F}]+ \mathbb{R}[\mathscr{G}]$. 
\end{claim}

\begin{proof}[Proof of Claim \ref{claim7}]
For a general point $x \in X$, Lemma \ref{dim} gives: 
$${\rm dim}\,{\rm Locus}(\mathscr{G}, \mathscr{F};x) \ge 4+2-2=4,$$
so this locus coincides with $G$ by Claim \ref{claim5}. 
Hence, the assertion follows from Lemma \ref{num}. 
\end{proof}

\begin{claim}\label{claim6}
$E_R\cdot \mathscr{G} = 0$. 
\end{claim}

\begin{proof}[Proof of Claim \ref{claim6}]
First, we assume  $E_R\cdot \mathscr{G} < 0$. 
Then $G \subset E_R$. 
Since a general point $x \in X$ and $G$ can be connected by a curve parametrized by $\mathscr{F}$, this implies $E_R\cdot \mathscr{F} >0$, a contradiction. 

Next, we assume $E_R\cdot \mathscr{G} > 0$. 
Since $X={\rm Locus}(\overline{\mathscr{F}};G)$, there exists a reducible cycle $C+C'$ parametrized by $\overline{\mathscr{F}}$ such that 
\begin{itemize}
\item $C$ intersects $G$,
\item neither $[C]$ nor $[C']$ belongs to $\mathbb{R}[\mathscr{F}]+\mathbb{R}[\mathscr{G}]+(\pm R)$, where $\pm R:=R \cup(-R)$. 
\end{itemize}
Indeed, otherwise $N_1(X)$ must be generated by $[\mathscr{F}]$, $[\mathscr{G}]$, $R$ according to Lemma \ref{num} and Claim \ref{claim7}. 
Let $\mathscr{U}$ be a family of deformations of $C$. 
Notice that $\mathscr{U}$ is unsplit. 

Since $E_R\cdot \mathscr{G} > 0$, $E_R$ intersects ${\rm Locus}(\mathscr{G};C)$, so we get a family $\mathscr{H}$ of rational curves such that $[\mathscr{H}] \in R$ and $H:={\rm Locus}(\mathscr{H})$ intersects ${\rm Locus}(\mathscr{G};C)$. 
Suppose that $\mathscr{H}$ has minimal anti-canonical degree among such families. 
This implies that $\mathscr{H}$ is unsplit on ${\rm Locus}(\mathscr{G};C)$. 
Note that $\mathscr{H}$ and $\mathscr{G}$ are numerically independent because $E_R\cdot \mathscr{G} > 0$ and $E_R\cdot \mathscr{H} < 0$. 
Hence, we see that $\mathscr{H}$ is unsplit as in Case \ref{case2} of Proposition \ref{prop3}. 

If ${\rm dim}\,H=4$, then $\rho_X=3$ by Lemma \ref{lem1}(1) for $(\mathscr{H}_1, \mathscr{H}_2, \mathscr{H}_3)=(\mathscr{H}, \mathscr{G}, \mathscr{U})$, a contradiction. 
However, even if ${\rm dim}\,H=5$, then $H=E_R$, so $H\cdot \mathscr{G}>0$. 
This implies $\rho_X=3$ by Lemma \ref{lem1}(2) for $(\mathscr{H}_1, \mathscr{H}_2, \mathscr{H}_3)=(\mathscr{H}, \mathscr{G}, \mathscr{U})$, also a contradiction. 
\end{proof}
Thus we obtain the conclusion of Proposition \ref{prop4}. 
\end{proof}

\begin{cor}[{Theorem \ref{main}(2) for the case (b2)}]
Let $X$, $\mathscr{F}$, and $\mathscr{G}$ be as in Assumption \ref{Ass}. 
Assume that every extremal ray of $X$ is divisorial. 
Then $\rho_X \le 4$ holds. 
\end{cor}

\begin{proof}
If $\rho_X \ge 5$, then Proposition \ref{prop4} yields that $N_1(X)$ is generated by $[\mathscr{W}_1],$ $\ldots ,$ $[\mathscr{W}_{\rho-2}]$, that is a contradiction. 
\end{proof}

\end{document}